\documentclass{amsart}
\usepackage[utf8]{inputenc}
\usepackage{amssymb, tipa}
\usepackage{graphicx}
\usepackage{xcolor}
\usepackage{tikz}
\usepackage{tikz-cd}
\usepackage{csquotes}
\usepackage[english]{babel}
\usepackage[fixlanguage, datename]{babelbib}
\usepackage{hyperref}
\usepackage{amsthm}
\usepackage{amsfonts}
\usepackage{mathtools}
\usepackage{stmaryrd}
\usepackage{MnSymbol}

\newtheorem{theorem}{Theorem}[section]

\newtheorem{corollary}[theorem]{Corollary}
\newtheorem{fact}[theorem]{Fact}

\theoremstyle{definition}
\newtheorem{definition}[theorem]{Definition}

\newtheorem{question}[theorem]{Question}

\newtheorem*{main}{Main Theorem}

\def\R{\mathbb{R}}

\def\N{\mathbb{N}}
\newcommand{\FF}{\mathrm{fields}}

\newcommand{\cL}{\mathcal{L}}
\newcommand{\dcl}{\mathrm{dcl}}
\newcommand{\acl}{\mathrm{acl}}

\newcommand{\alg}{\mathrm{alg}}
\newcommand{\ld}{\mathrm{ld}}
\newcommand{\ACF}{\mathrm{ACF}}
\newcommand{\ACVF}{\mathrm{ACVF}}
\newcommand{\RCF}{\mathrm{RCF}}

\newcommand{\M}{\mathbb M}
\newcommand{\A}{\mathbb A}

\newcommand{\ldfree}[1]{{\underset{#1}{\downfree^{\ld}}}}
\newcommand{\aclfree}[1]{{\underset{#1}{\downfree^{\acl}}}}
\newcommand{\algfree}[1]{{\underset{#1}{\downfree^{\alg}}}}

\DeclareMathOperator{\tp}{tp}

\title{Quantifier elimination for lovely pairs of strongly geometric fields}

\author[Pablo Cubides Kovascics]{Pablo Cubides Kovacsics}
\address{Pablo Cubides Kovacsics, Departamento de Matemáticas, Universidad de los Andes, 
Carrera 1 \# 18A - 12, 
111711, Bogot\'a, Colombia}
\email{p.cubideskovacsics@uniandes.edu.co}
\thanks{The first author was supported by the research grant FAPA (INV-2023-158-3157) \emph{New methods in non-archimedean geometry and its applications}, funded by Universidad de los Andes.}

\author[Felipe Estrada]{Felipe Estrada}
\address{Felipe Estrada, Departamento de Matemáticas, Universidad de los Andes, 
Carrera 1 \# 18A - 12, 
111711, Bogot\'a, Colombia}
\email{f.estrada@uniandes.edu.co}
\thanks{}

\author[Juan Pérez]{Juan Pérez}
\address{Juan Pérez, Departement de mathématiques, UMONS, Place du parc 20, Mons, Belgium}
\email{260033@umons.ac.be}
\thanks{The third author was supported by Institute Complexys at UMONS}

\author[David Rincón]{David Rincón}
\address{David Rincón, The University of Manchester, School of Mathematics, Oxford Road, Manchester
M13 9PL, UK}
\email{david.rincon@postgrad.manchester.ac.uk}
\thanks{}

\subjclass[2020]{03C60, 12L12}
\keywords{Geometric fields, very slim fields, lovely pairs.}

\begin{document}

\maketitle

\begin{abstract}
Let \(T\) be a complete strongly geometric theory of fields with quantifier elimination. We show that the theory of lovely pairs of $T$ has quantifier elimination in  Delon’s definitional expansion by predicates for linear independence and function symbols for the corresponding coordinate functions. Apart from recovering Delon's original results for pairs of algebraically closed fields and dense pairs of algebraically closed valued fields, we obtain as particular cases, quantifier elimination for theories of dense pairs of real closed and $p$-adically closed fields.
\end{abstract}

\

The study of pairs of models of a complete theory is a classical topic in model theory, dating back to results of A. Robinson \cite{Robinson} on pairs of algebraically closed and real closed fields. In the present text we study pairs of \emph{strongly geometric fields}, that is, expansions of fields for which, in all elementary extensions, the model-theoretic algebraic closure coincides with the field-theoretic algebraic closure (also called \emph{very slim} by M. Junker and J. Koenigsmann in \cite{JK}). W. Johnson and J. Ye proved in \cite{johnson-ye}, that the theory of any such field is a geometric theory, in the sense that it eliminates the quantifier $\exists^\infty$ and the model-theoretic algebraic closure has the exchange property. Examples of strongly geometric fields include algebraically closed and real closed fields, perfect PAC-fields and henselian valued fields of characteristic 0. 

The theory of pairs of strongly geometric fields we are interested in is the theory of their \emph{lovely pairs}, a notion introduced by A. Berenstein and E. Vassiliev in \cite{beren-vassi} encompassing various well-studied theories of pairs. Their definition will be recalled in Section 2. Examples include classical theories such as the theories of pairs of algebraically closed fields, dense pairs of real closed (resp. $p$-adically closed) fields and dense pairs of algebraically closed valued fields. 

In \cite{delon-pairs}, F. Delon proved  that after adding predicates for linear independence and function symbols for the associated coordinate functions, pairs of algebraically closed fields and dense pairs of algebraically closed valued fields admit quantifier elimination. We will call such a definitional expansion \emph{Delon's expansion} (its formal definition will be given in Section 2). It turns out that the theories of pairs studied by Delon are precisely theories of lovely pairs. Our main theorem generalizes Delon's results by showing that if $T$ is the theory of a strongly geometric field with quantifier elimination, then the corresponding theory of lovely pairs of models of $T$ has quantifier elimination in Delon's expansion. In light of our result, one can conclude that the heart of Delon's quantifier elimination theorems lies precisely in the (strongly) geometric nature of the underlying theory of fields. As other particular cases, we obtain quantifier elimination for the theories of dense real closed (resp. $p$-adically closed) fields, and lovely pairs of henselian valued fields of characteristic 0 such as $\R(\hspace{-.075cm}(t)\hspace{-.075cm})$ and $\mathbb{C}(\hspace{-.075cm}(t)\hspace{-.075cm})$, viewed as structures in Macintyre's language after addition of constants (see later Corollary \ref{cor:geoFields}).

The article is laid out as follows. In Section 1 we present the setting, recall the main definitions involved and gather a couple of algebraic lemmas which will be later used. Finally, in Section 2 we prove our main theorem and derive some consequences of it.  

\section{Setting and preliminaries}

\subsection{Lovely pairs of strongly geometric fields}

For convenience, we will work in the language of fields $\cL_{\FF}=\{+,-,\cdot, ^{-1},0,1\}$ (instead of the usual language of rings). For $\cL\supseteq\cL_\FF$, let $T$ be the $\cL$-theory of a strongly geometric field. Recall that this means that in any model of $T$, the model-theoretic algebraic closure coincides with the field-theoretic algebraic closure. By \cite[Theorem 2.5]{johnson-ye}, this implies that $T$ is geometric.\footnote{Note that there are expansions of fields which are not strongly geometric but do have an underlying geometric theory as, for example, $\mathrm{Th}(\R,+,\cdot,\exp)$. Surprisingly, \cite[Example 4.3]{johnson-yeII} shows that there are pure fields which are not strongly geometric but have an underlying geometric theory.}   

\

Lovely pairs of geometric structures (i.e., their theory is geometric) were introduced by A. Berenstein and E. Vassiliev in \cite{beren-vassi}. We recall in this section their definition and their main properties. 

Let $T$ be a complete geometric $\cL$-theory and $M$ be a model of $T$. For $A, B, C\subseteq M$ we write $C\aclfree{A} B$ whenever if $C$ is $\acl$-independent over $A$, it remains $\acl$-independent over $A\cup B$. Given a type $q \in S(B)$,
we say that $q$ is \emph{free over $A$} if $c\aclfree{A} B$ for each $c\models q$.   

We let $\cL_P$ be the language $\cL$ extended by a unary predicate $P$. By a pair of $\cL$-structures we mean an $\cL_P$-structure $\M=(M, P_M)$ such that $P_M\leqslant_{\cL} M$. For $A\leqslant_{\cL} M$ we let $P_A=P_M\cap A$ and write $\A$ for $(A,P_A)$. Let $T'$ be the $\cL_P$-theory of pairs $(M,P_M)$ such that $M\models T$ and $\acl_\cL(P_M)=P_M$. 

\begin{definition}\label{def:lovelyPair} 
Let $\kappa$ be a cardinal greater than $|T|$.
A pair of $\cL$-structures $(M, P_M)$ is a \emph{$\kappa$-lovely pair of models of $T$} if
\begin{enumerate}
    \item $(M,P_M)\models T'$; 
    \item (Coheir property) If $A\subseteq M$ with $|A|<\kappa$ and $q \in S_1(A)$ is free over $P_A$, there is $a \in P_M$ such that $a\models q$;
\item  (Extension property) If $A\subseteq M$ with $|A|<\kappa$ and $q\in S_1(A)$, there is
$a \in M$ such that $a \models q$ and $\tp_{\cL}(a/A\cup P_M)$ is free over $A$.
\end{enumerate}
A pair of $\cL$-structures $(M, P_M)$ is a lovely pair if it is $\kappa$-lovely pair for some cardinal $\kappa>|T|$. 
\end{definition}

That the previous definition is equivalent to \cite[Definition 2.3]{beren-vassi}, follows from the discussion after \cite[Lemma 2.4]{beren-vassi}. The following summarizes the main results from \cite{beren-vassi} that we will need: 

\begin{theorem}[{\cite[Lemma 2.5, Corollary 2.9 and Theorem 2.10]{beren-vassi}}]\label{thm:lovely-pairs} For any cardinal $\kappa>|T|$, the $\kappa$-lovely pairs of models of $T$ exist and they are elementary pairs. Their common theory $T_P$ is complete. In addition, if $(M, P_M)$ is a $\kappa$-saturated model of $T_P$ for a cardinal $\kappa>|T|$, then $(M, P_M)$ is a $\kappa$-lovely pair.
\end{theorem}

In what follows we will assume that $T$ is a complete strongly geometric $\cL$-theory of fields. We will study the corresponding theory of lovely pairs. 

\subsection{Delon's expansion}\label{sec:delon}

We extend $\cL_P$ to the language $\cL_D$, \emph{Delon's expansion}, defined by 
\[
\cL_D=\cL_P\cup \{\ell_n : n\in \N_{\geqslant 1}\} \cup \{f_{n,i} : n\in \N_{\geqslant 1}, 1\leqslant i\leqslant n\}
\]
where $\ell_n$ is a predicate symbol of arity $n$ and $f_{n,i}$ is a function symbol of arity $n+1$. Let $T$ be an $\cL$-theory of fields and let $T_2$ be any $\cL_P$-theory of pairs $\M=(M,P_M)$ such that $ M\models T$ and $P_M \leqslant_\cL M$. We let $T_2(D)$ be the $\cL_D$-theory $T_2$ together with axioms capturing in any model $(M,P_M)\models T_2(D)$ that
\begin{align*}
& \ell_n(x_1,\ldots, x_n) \Leftrightarrow x_1, 
 \ldots, x_n \text{ are linearly independent over $P_M$} \\
& f_{n,i}(y,x_1,\ldots, x_n)=
\begin{cases}
z_i & \text{ if  $\ell_n(x_1,\ldots, x_n) \wedge y=\sum_{i=1}^n z_ix_i \wedge \bigwedge_{i=1}^n P(z_i)$ } \\
0 & \text{ otherwise.}
  \end{cases} 
\end{align*}

\subsection{Auxiliary results}

\begin{definition}\label{def:1} Let $k\subseteq K, L$ be fields all contained in some bigger field. 
\begin{enumerate}
    \item The field $K$ is \emph{linearly disjoint from $L$ over $k$}, denoted by $K\ldfree{k} L$, if whenever a subset $A$ of $K$ is $k$-linearly independent, it is also $L$-linearly independent. 
    \item The field $K$ is \emph{free from $L$ over $k$}, denoted by $K\algfree{k} L$, if whenever a subset $A$ of $K$ is $k$-algebraically independent, it is also $L$-algebraically independent.
    \item The extension $k\subseteq K$ is regular if and only if the extension is separable and $k^\alg\cap K = k$ (in particular, every elementary extension of fields is regular).   
\end{enumerate}
\end{definition}

We recall some facts about linear disjointness that we will be later used (see \cite{Lang}): 

\begin{fact}\label{fact.disln} Let $k\subseteq K, L$, and $F$ be fields all contained in some bigger field. 
\begin{enumerate}
    \item The relation $\ldfree{k}$ is symmetric.
    \item (Transitivity) $K \ldfree{k} L$ if and only if $K \ldfree{k} F$ and $KF \ldfree{F} L$, where $k \subseteq F\subseteq L$.  
    \item $K \ldfree{k} L$ implies $K \algfree{k} L$.
    \item A field extension $K\subseteq L$ is \emph{separable} if and only if $K\ldfree{K^p} L^p$.
    \item If $K \ldfree{k} L$ and $k \subseteq K$ is regular, then $L\subseteq LK$ is regular.
\end{enumerate}
\end{fact}

The following fact summarizes various results in \cite{delon-pairs}.

\begin{fact}[Delon]\label{fact:substructure}
Let $\M=(M, P_M)$ and $\M'=(M',P_{M'})$ be pairs of $\cL$-structures (i.e., $P_M\leqslant_{\cL} M$). 
\begin{enumerate}
\item Let $A\leqslant_{\cL} M$. Then 
\[
\A\leqslant_{\cL_D} \M \Leftrightarrow A \ldfree{P_A} P_M
\]
\item Let $A\subseteq M$ be such that $P_A$ is an $\cL$-substructure of $P_M$, and let $\sigma\colon A\to M'$ be a partial $\cL$-isomorphism. Let $A'=\sigma(A)$. Then $\sigma$ is a partial $\cL_D$-isomorphism if and only if
\[
\sigma(P_A)=P_{A'} \hspace{1cm} A  \ldfree{P_A} P_M \hspace{1cm} A' \ldfree{P_{A'}} P_{M'}. 
\]
\end{enumerate}
\end{fact}

\section{Relative quantifier elimination}

Through this section $T$ denotes a complete strongly geometric $\cL$-theory of fields with quantifier elimination. As above, we let $T_P$ be the theory of lovely pairs of models of $T$. Let $T_D=T_P(D)$ as defined in Section \ref{sec:delon}. 

\begin{main}\label{thm:main}
The theory $T_D$ has quantifier elimination. 
\end{main}

\begin{proof} We use Shoenfield's criterion for quantifier elimination. Let $\M=(M,P_M)$ and $\M'=(M',P_{M'})$ be two models of $T_D$ where $|M|=|T|=\kappa$ and $\M'$ is $\kappa^+$-saturated. In particular, by Theorem \ref{thm:lovely-pairs}, $\M'$ is a $\kappa^+$-lovely pair of models of $T$. Let $\A=(A,P_A)\leqslant_{\cL_D} \M$ and $\A'=(A',P_{A'})\leqslant_{\cL_D} \M'$ be $\cL_{D}$-substructures and $\sigma\colon A\to A'$ be an $\cL_D$-isomorphism. We need to extend $\sigma$ to an $\cL_D$-embedding from $\M$ to $\M'$. We will proceed by extending $\sigma$ as indicated in the following steps: 

\medskip

\emph{Step 1:} Let $X=\{e_\alpha\in P_M : \alpha<\kappa\}$ be a maximal $\acl_\cL$-independent set over $P_A$. Letting $A_{\alpha} = A\cup \{e_\beta: \beta<\alpha\}$ for each $\alpha<\kappa$, we build by induction a nested sequence of partial $\cL$-isomorphisms $\sigma_\alpha\colon A_\alpha\to M'$ sending $e_{\beta}\to e_{\beta}'$ where $e'_{\beta}\in P_{M'}$ for each $\beta<\alpha$. Set $\sigma_0 = \sigma$. Assume that $\sigma_\alpha$ has been defined and let $A_\alpha'$ be the image of $\sigma_\alpha$. Consider the type
\[	
q(z) = \{\varphi(z,\sigma_\alpha(a)) : M \models \varphi(e_{\alpha}, a), \varphi \text{ an $\cL$-formula, } a \text{ a tuple in } A_\alpha \}. 
\] 
Note that $q\in S_\cL(A_\alpha')$. Indeed, maximality is clear and consistency follows from the fact that $\sigma_\alpha$ is a partial $\cL$-elementary embedding (which follows by quantifier elimination of $T$). We show that $q$ is free over $P_{A_\alpha'}$. For suppose it is not and let $\varphi(z,\sigma_\alpha(a))$ be a formula such that $M'\models \exists^{\leqslant n} z \varphi(z, \sigma_\alpha(a))$ with $a$ a tuple in $A_\alpha$. Then, since $\sigma_\alpha$ is partial $\cL$-elementary, we have that $M\models \exists^{\leqslant n} z \varphi(z, a)$. Since $M \models \varphi(e_{\alpha}, a)$, by the strongly geometric assumption, we obtain that $e_{\alpha}$ is algebraic over $A_\alpha$. Since $\A\leqslant_{\cL_D} \M$, by Facts \ref{fact.disln} and \ref{fact:substructure} we have the following implications
\[
A  \ldfree{P_A} P_M \Rightarrow A\acl_\cL(P_{A_\alpha})  \ldfree{\acl_\cL(P_{A_\alpha})} P_M \Rightarrow A\acl_\cL(P_{A_\alpha})  \algfree{\acl_\cL(P_{A_\alpha})} P_M. 
\]
However, the last part contradicts that $e_\alpha$ is algebraic over $A_\alpha$ since $e_\alpha\notin\acl_\cL(P_{A_\alpha})$ and $A_\alpha\subseteq A\acl_\cL(P_{A_\alpha})$. Therefore, by the coheir property of lovely pairs, we define $e_{\alpha}'\in P_{M'}$ as some (any) element realizing $q$. By the definition of $q$, $\sigma_{\alpha+1}$ is a partial $\cL$-isomorphism preserving $P$. For $\alpha$ a limit ordinal, one defines $\sigma_\alpha$ as the union of all $\sigma_\beta$ for $\beta<\alpha$. Since the union of partial isomorphism is a partial isomorphism, we are done.

\medskip

\emph{Step 2:} Call $\sigma\colon A\cup X\to A'\cup X'$ be the map obtained in Step 1 where $X'=\{e_\alpha' : \alpha<\kappa\}$. We extend $\sigma$ to a partial $\cL$-isomorphism $\sigma' \colon A\cup P_M \to A' \cup \acl_\cL(X')$. Note that $P_M=\acl_\cL(X\cup P_A)$, so we only need to extend the map to elements in $B = \acl_\cL(X\cup P_A)\setminus (X\cup P_A)$. For each element $b\in B$, let $\psi_b(z)$ be an $\cL(X\cup P_A)$-formula for which $\psi_b(b)$ holds and $|\psi_b(P_M)|$ is minimal. For $b\in B$, the set $\psi_b(P_M)$ is unique, and since $\sigma$ is a partial $\cL$-elementary map, it is in bijection with $\psi_b^\sigma(P_{M'})$ (where $\psi_b^\sigma$ is the formula obtained after applying $\sigma$ to the parameters of $\psi_b$). We extend $\sigma$ to some (any) bijection $\sigma'$ sending, for each $b\in B$, the $\acl_\cL$-orbit $\psi_b(P_M)$ to the $\acl_\cL$-orbit $\psi_b^\sigma(P_{M'})$. The map $\sigma'$ clearly preserves the predicate, so we only need to show that it is a partial $\cL$-isomorphism. This is a standard argument. However, since it will later reappear, we give it once for the reader's convenience. Reasoning by contradiction, let $\varphi(x,y,z)$ be an $\cL$-formula such that 
\[
M\models \varphi(b, a, e) \hspace{1cm} \text{ and } \hspace{1cm} M'\models \neg\varphi(b' ,\sigma(a),\sigma(e))
\]  
where $b=(b_1,\ldots, b_n)\in B^n$, $a$ is a tuple in $A\setminus (X\cup P_A)$, $e$ a tuple in $X\cup P_A$ and $b'$ the image of $b$ under $\sigma'$. Possibly enlarging the tuple $e$, let $\theta(b,a,e)$ be the formula
\[
\varphi(b, a, e) \wedge \bigwedge_{i=1}^n \psi_{b_i}(b_i),  
\]  
which holds in $M$. Since $\sigma$ is $\cL$-elementary, we have that $M\models \exists z \theta(z, a, e)$ implies $M'\models \exists z \theta(z,\sigma(a),\sigma(e))$. Thus
\[
M'\models \neg \varphi(b', \sigma(a), \sigma(e)) \wedge \bigwedge_{i=1}^n \psi_{b_i}(b_i') \wedge \exists z \theta(z,\sigma(a),\sigma(e)), 
\]  
which contradicts the minimality of $|\psi_{b_1}(P_{M'})|$. Therefore, $\sigma'$ is a partial $\cL$-isomorphism.

\medskip

\emph{Step 3:} Call $\sigma\colon A\cup P_M \to A'\cup C$ the map from Step 2 where $C$ is the image of $P_M$ under $\sigma$. By the definition of $C$ and the transitivity of linear disjointness we have that 
\[
A\cup P_M \ldfree{P_M} P_M \hspace{1cm} A'\cup C \ldfree{C} P_{M'}.
\]
Therefore, by Fact \ref{fact:substructure}, $\sigma$ is a partial $\cL_D$-isomorphism. We extend $\sigma$ to the unique partial $\cL_D$-isomorphism $\sigma'\colon \dcl_\cL(A\cup P_M)\to \dcl_\cL(A'\cup C)$ (alternatively, one can take  $\dcl_{\mathcal{L}_D}$).  

\

\emph{Step 4:} By Step 3, we are in a situation where $\sigma\colon \A\to\A'$ is an $\cL_D$-isomorphism between $\cL_D$-substructures and $P_A=P_M$. Let $Y= \{b_\alpha : \alpha<\lambda\leqslant\kappa\}$ be an $\acl_\cL$-independent subset of $M$ over $A$. Similarly to Step 2, we extend $\sigma$ to a partial $\cL$-isomorphism preserving the predicate by constructing a nested sequence of partial $\cL$-isomorphisms $\sigma_\alpha\colon A_\alpha \to A_\alpha'$ where $A_\alpha = A\cup \{b_\beta: \beta<\alpha\}$ setting $\sigma_0=\sigma$. Suppose that $\sigma_\alpha$ has been defined. Consider the type  
\[	
q(z) = \{\varphi(z,\sigma_\alpha(a)) : M \models \varphi(b_{\alpha}, a), \varphi \text{ an $\cL$-formula, } a \text{ a tuple in } A_\alpha \}. 
\] 
As in Step 1, note that $q\in S_\cL(A_\alpha')$. Hence, by the extension property of lovely pairs, there is $b_{\alpha}'\in M'$ realizing $q$ such that $\tp(b_{\alpha}' / A_\alpha'\cup P_{M'})$ is free over $A_\alpha'$. Note that since $b_{\alpha}'\in M'$ realizes $q$, we have that $b_{\alpha}' \notin \acl_\cL(A_\alpha')$ (same argument as in Step 1). Therefore, by freeness, $b_{\alpha}'\notin \acl_\cL(A_\alpha'\cup P_{M'})$. Extending $\sigma_\alpha$ by sending $b_{\alpha}$ to $b_{\alpha}'$ gives a partial $\cL$-isomorphism preserving $P$. For $\alpha$ a limit ordinal we just take unions. Abusing of notation, let $\sigma\colon A\cup Y\to A'\cup Y'$ denote the map we just constructed where $Y'$ is the image of $Y$. Note that $A'\cup Y'$ is linearly disjoint from $P_{M'}$ over $P_{A'}$. Indeed, suppose $D:=\{a_1,\ldots, a_m, b_1',\ldots, b_n'\}$ is linearly dependent over $P_{M'}$, where $a_1,\ldots, a_m\in A'$ and $b_1',\ldots, b_n'\in Y'$. Then, we must have that $n=0$, otherwise we contradict that $Y'$ is $\acl_\cL$-independent over $A'$ (defined as such by construction). Since there are no $b_i'$, the set $D$ is contained in $A'$ and is linearly dependent over $P_{M'}$, so it is over $P_{A'}$ by Fact \ref{fact:substructure}.  

\medskip
\emph{Step 5:} Let $\sigma\colon A\cup Y\to A'\cup Y'$ be the map from Step 4. We extend $\sigma$ to a partial $\cL$-isomorphism $\sigma'$ from $M=\acl_\cL(A\cup Y)$ to $\acl_\cL(A'\cup Y')$. Again, we do so by sending $\acl_\cL$-orbits to $\acl_\cL$-orbits bijectively as in Step 2. Let us show this must preserve the predicate and that $(\acl_\cL(A'\cup Y'), P_{\acl_{\cL}(A'\cup Y')})$ is an $\cL_D$-substructure.\\
For the former, by Step 4, $A'\cup  Y'$ is $\acl_\cL$-independent from $P_{M'}$ over $P_{A'}$. Since $P_{A'}$ is $\acl_\cL$-closed, we have that $\acl_\cL(A'\cup Y')\cap P_{M'}=P_{A'}$, otherwise  there is $e \in \acl_{\cL}(A'\cup Y')\cap P_{M'}\setminus P_{A'}$,  so by Fact \ref{fact:substructure} the element $e$ must be independent over $A'\cup Y'$. Thus, every partial $\cL$-isomorphism from $\acl_{\cL}(A\cup Y)$ to $\acl_{\cL}(A'\cup Y')$ extending $\sigma$  preserves the predicate.\\
It remains to show that $\acl_{\cL}(A'\cup Y') \ldfree{P_{A'}} P_{M'}$. By quantifier elimination in $T$, $P_{A'}\preccurlyeq_\cL  P_{M'}$, so Fact \ref{fact.disln} applies and $ P_{A'}\subset P_{M'}$ is regular. From Step 4 we have that $A'(Y') \ldfree{P_{A'}} P_{M'}$, implying that the extension $A'(Y')\subseteq A'(Y')P_{M'}$ is also regular. Using Fact \ref{fact.disln}, the regularity of the extension $ A'(Y')\subseteq A'(Y')P_{M'}$ and the algebraicity of the extension $ A'(Y')\subseteq \acl_{\cL}(A'(Y'))$ (here we use that $T$ is strongly geometric) yields 
\[
\acl_{\cL}(A'(Y')) \ldfree{A'(Y')} A'(Y')P_{M'}.
\]
Applying transitivity we obtain that $\acl_{\cL}(A'\cup Y') \ldfree{P_{A'}} P_{M'}$.

We conclude by Fact \ref{fact:substructure}. 
\end{proof}

The following corollary provides examples of our main result. Points (1) and (3) generalize Delon's results from \cite{delon-pairs}. 

\begin{corollary}\label{cor:geoFields}
Let $T$ be one of the following $\cL$-theories ($\cL$ depending on $T$): 
\begin{enumerate}
    \item $\ACF_p$ in the language of rings/fields;
    \item $\RCF$ in the language of ordered rings/fields;
    \item any completion of $\ACVF$ in any natural one-sorted language of valued fields;
    \item the theory of a henselian field of characteristic 0 in a definitional expansion of the language of rings with quantifier elimination;
    \item the theory of a $p$-adically closed valued field or of a field of Laurent series such as $\R(\hspace{-.075cm}(t)\hspace{-.075cm})$ and $\mathbb{C}(\hspace{-.075cm}(t)\hspace{-.075cm})$ (in Macintyre's language with enough constants);
\item the theory of a perfect PAC-field in a definitional expansion of the language of rings with quantifier elimination. 
\end{enumerate}
Then the $\cL_D$-theory $T_D$ of lovely pairs of $T$ has quantifier elimination.  
\end{corollary}

\begin{proof}
The result follows directly from Theorem \ref{thm:main} after noticing that such fields are strongly geometric and have quantifier elimination in the given language. That henselian fields of characteristic 0 are strongly geometric (in the language of rings) follows from \cite[Theorem 5.5]{JK}. Note that by relative quantifier elimination, any henselian valued field of equicharacteristic 0 is strongly geometric in any natural language of valued fields.  
\end{proof}

We finish with the following question: 

\begin{question}
Can our Main Theorem be extended to the context of Mordell-Lang pairs as defined in \cite{BHK}?    
\end{question}

\bibliographystyle{babplain}
\bibliography{ref}

\end{document}